\documentclass[11pt,twoside, final]{amsart}
\copyrightinfo{0}{Iranian Mathematical Society}
\pagespan{1}{\pageref*{LastPage}}
\usepackage{etoolbox,lastpage}
\commby{}
\date{\scriptsize   Received: , Accepted: .}
\usepackage{amsmath,amsthm,amscd,amsfonts,amssymb,enumerate}
\usepackage{graphicx}		
\usepackage{color}
\usepackage[colorlinks]{hyperref}
\usepackage{cleveref}
\usepackage{tikz} 
\usetikzlibrary {positioning}
\tikzstyle{every node}=[circle, draw, fill=black,inner sep=0pt, minimum width=4pt,
node distance =1 cm and 1cm ]

\newtheorem{theorem}{Theorem}[section]

\newtheorem{lemma}[theorem]{Lemma}
\newtheorem{fact}[theorem]{Fact}
\newtheorem{corollary}[theorem]{Corollary}
\theoremstyle{definition}
\newtheorem{definition}[theorem]{Definition}

\newtheorem{notation}[theorem]{Notation}
\theoremstyle{remark}
\newtheorem{remark}[theorem]{Remark}
\numberwithin{equation}{section}
\DeclareMathOperator{\dist}{dist} 
 
\newenvironment{mytikzpicture}{\begin{center}\vspace*{15pt}\begin{tikzpicture}}
{\
\end{tikzpicture}\vspace*{15pt}\end{center}
}

\DeclareMathOperator{\diag}{Diag}
\DeclareMathOperator{\theory}{Th}

\DeclareMathOperator{\age}{Age}
\DeclareMathOperator{\cl}{cl}
\DeclareMathOperator{\acl}{acl}
\DeclareMathOperator{\dcl}{dcl}
\DeclareMathOperator{\type}{tp}

\DeclareMathOperator{\cltype}{cltp}

\DeclareMathOperator{\dMe}{d}

\DeclareMathOperator{\idot}{i}
\DeclareMathOperator{\cMe}{c}

\DeclareMathOperator{\clforms}{CLF}

\DeclareMathOperator{\univ}{UNIV}
\DeclareMathOperator{\Nr}{Nr}
\DeclareMathOperator{\eq}{eq}
\DeclareMathOperator{\aut}{\mathbb{A}ut}

\DeclareMathOperator{\valc}{VAL}
\DeclareMathOperator{\val}{val}

\newcommand{\kplusomega}{\mathcal{K}_{\omega}}
\newcommand{\kplusomegabar}{\overline{\mathcal{K}_{\omega}}}
\newcommand{\kplusalpha}{\mathcal{K}_{\alpha}}
\newcommand{\kplusalphabar}{\overline{\mathcal{K}_{\alpha}}}
\newcommand{\kplusn}{\mathcal{K}_{n}}
\newcommand{\kplusnbar}{\overline{\mathcal{K}_{n}}}
\newcommand*{\freeJoin}[3]{#1\sqcup_{#2}#3}

\newcommand{\closed}[2][A]{#1\leq^{*}#2}
\newcommand{\minpair}[2][A]{#1\not\leq^{*}_{\min}#2}

\newcommand{\intrext}[2][A]{#1\leq^{*}_{\idot}#2}
\newcommand{\intrextw}[2][A]{#1\not\leq_{\idot}#2}
\newcommand{\fraisse}{Fra{\"i}ss{\'e}\hspace{4pt}}

\newcommand*{\card}[1][M]{\vert #1\vert}
\newcommand{\subsetfinite}{\subseteq_{\omega} }

\newcommand{\forkindepPrv}[1]{
	\mathrel{
		\mathop{
			\vcenter{
				\hbox{\oalign{\noalign{\kern-.3ex}\hfil$\vert$\hfil\cr
						\noalign{\kern-.7ex}
						$\smile$\cr\noalign{\kern-.3ex}}}
			}
		}\displaylimits_{#1}
	}
}
\newcommand{\forkindep}[3]{#1\forkindepPrv{#2}#3}
\newcommand{\indepPrv}[2][\Gamma]{
	\mathrel{
		\mathop{
			\vcenter{
				\hbox{{\oalign{\noalign{\kern-.3ex}\hfil$\hspace*{4pt}\vert^{#1} $\hfil\cr
							\noalign{\kern-.7ex}
							$ \smile $\cr\noalign{\kern-.3ex}}}}
			}
		}\displaylimits_{#2}
	}
}
\newcommand{\indep}[4][\Gamma]{#2\indepPrv[#1]{#3}#4}
\newcommand*{\hrchy}[2][s]{
	\ifthenelse{\equal{#1}{s}}
		{\Sigma\overset{^{\raisebox{1pt}{\tiny c}}}{_{\raisebox{-1pt}{\tiny #2}}}}
		{\ifthenelse{\equal{#1}{p}}
			{\Pi\overset{^{\raisebox{1pt}{\tiny c}}}{_{\raisebox{-1pt}{\tiny #2}}}}
			{\ifthenelse{\equal{#1}{h}}
				{\mathbf{H}\overset{^{\raisebox{1pt}{\tiny c}}}{_{\raisebox{-1pt}{\tiny #2}}}}
				{bbb}
			}
		}
}
\newcommand*{\hrchyLong}[2][s]{
	\ifthenelse{\equal{#1}{s}}
		{\Sigma\overset{\mkern -15mu ^{\raisebox{1pt}{\tiny c}}}{_{\raisebox{-1pt}{\tiny #2}}}}
		{\ifthenelse{\equal{#1}{p}}
			{\Pi\overset{\mkern -15mu ^{\raisebox{1pt}{\tiny c}}}{_{\raisebox{-1pt}{\tiny #2}}}}
			{\ifthenelse{\equal{#1}{h}}
				{\mathbf{H}\overset{\mkern -15mu ^{\raisebox{1pt}{\tiny c}}}{_{\raisebox{-1pt}{\tiny #2}}}}
				{bbb}
			}
		}
}

\newcommand*{\hrchySeq}[3][s]{
	\ifthenelse{\equal{#1}{s}}
		{\Sigma\overset{^{\raisebox{1pt}{\tiny $ \cMe_{#3} $ }}}{_{\raisebox{-1pt}{\tiny #2}}}}
		{\ifthenelse{\equal{#1}{p}}
			{\Pi\overset{^{\raisebox{1pt}{\tiny  $ \cMe_{#3} $ }}}{_{\raisebox{-1pt}{\tiny #2}}}}
			{\ifthenelse{\equal{#1}{h}}
				{\mathbf{H}\overset{^{\raisebox{1pt}{\tiny $ \cMe_{#3} $ }}}{_{\raisebox{-1pt}{\tiny #2}}}}
				{bbb}
			}
		}
}
\newcommand{\proofAVName}{Claim}

\begin{document}

 
\title[Strict Superstablity and Decidability of Certain Generic Graphs]{Strict Superstablity and Decidability of Certain Generic Graphs} 
 
\author[A.N. ‌Valizadeh]{Ali N. Valizadeh}
\address[Ali N. Valizadeh]{Department of Mathematics and Computer Science, Amirkabir University of Technology,  Tehran, IRAN.}
\email{valizadeh.ali@aut.ac.ir}

\author[M. Pourmahdian]{Massoud Pourmahdian $^*$}
\address[M. Pourmahdian]{
Department of Mathematics and Computer Science, Amirkabir University of Technology,  Tehran, IRAN.}
\address{
School of Mathematics, Institute for Research in Fundamental Sciences (IPM),  Tehran, IRAN.}
\email{pourmahd@ipm.ir}

  \thanks{$^*$ Corresponding author}
%
 
 \maketitle
%

\begin{abstract}
We show that the Hrushovski-\fraisse limit of certain classes of trees lead to strictly superstable theories of various U-ranks. In fact, for each $ \alpha\in\omega+1\backslash\{0\} $ we introduce a strictly superstable theory of U-rank $ \alpha. $ Furthermore, we show that these theories are decidable and pseudofinite.
\end{abstract}
 
\section{\bf Introduction}\label{secIntro}
This paper introduces a variety of \textit{ultraflat} Hrushovski-\fraisse classes of acyclic graphs whose limits are strictly superstable and pseudofinite. A graph is called ultraflat if it does not contain any subgraph isomorphic to a graph obtained by adding new vertices on the edges of a fixed complete graph (\Cref{dfnComplGraphDivided,dfnUltraflat}). Since introducing the Hrushovski constructions in \cite{Hrushovski-NewStrongly}, several generalizations and investigations have been made on the subject. While Hrushovski's \textit{ab-initio} was intended to assemble a strongly minimal structure that refuted the Zilber's conjecture, various generalizations were seeking to find new examples in higher orders of the hierarchy of the classification theory. 

A thorough analysis of generic structures having stable theories appeared in \cite{Baldwin&Shi-StableGen} and continued by introducing a first order version of genericity, called semigenericity, in \cite{Baldwin&Shelah-Randomness} resulting in an axiomatization of the almost sure theory of random hypergraphs in which edge probability is defined using irrational powers less than 1. Later, the same notion of semigenericity was applied to other classes of finite structures in \cite{Pourmahd-SimpleGen} and \cite{Pourmahd-SmoothClasses} that led to a simple context in both first-order theories and non-elementary classes. There are other results obtained using Hrushovski \textit{ab-initio} constructions, a more remarkable among others was the introduction of an almost strongly minimal non-Desarguesian projective plane in \cite{Baldwin-NonDesarg} refuting other aspects of Zilber's conjecture.

The question of finding a strictly superstable \textit{ab-initio} generic structure was first asked in \cite[Question 12]{Baldwin-Pathological}. Positive answers to this question, using somewhat complicated constructions, were given by Ikeda and Kikyo in \cite{Ikeda-AbInitioSuperstable} and \cite{Ikeda&Kikyo-OnSuperStableGenerics}. In this paper, working with trees, we introduce a variety of strictly superstable generic structures whose ranks vary from $ 1 $ to $ \omega. $ 

The constructions given here fit naturally into the context of ultraflat graphs that are a particular well-behaved subclass of nowhere dense graphs (equivalently, superflat graphs \cite{Adler-nowherdense}). It is known that nowhere dense classes are tame from the view point of both stability and algorithmic model theory (\cite{Nesetril-Sparcity}, \cite{Podewski&Ziegler-StableGraphs}). In particular, every ultraflat graph is superstable (\cite{Herre-SuperstableGraphs}). 

On the other hand, applying finite Ehrenfeucht-\fraisse games that are well established for sparse random graphs enables us to prove decidability and pseudofiniteness for these structures.

\section{\bf Preliminaries}\label{secPrelim}
To fix notation, let $ \mathcal{L} $ be a finite language consisting only of a binary relation $ R. $ Unless stated otherwise, finite $ \mathcal{L} $-structures are shown by $ A, B, \ldots, $ and the possibly infinite structures are denoted by $ M, N, $ etc. For $ A,B\subseteq M, $ by $ AB $ we mean the structure that is induced from $ M $ on $ A\cup B. $ We also write $ A\subseteq_{\omega}M $ to mean that $ A $ is a finite substructure of $ M $ and by $ \age(M) $ we denote the set of all finite substructures of $ M. $

We define a function $ \delta, $ called the \textit{predimension} function, on all finite graphs, assuming $ R $ is symmetric and anti-reflexive, by letting
\[ \delta(A):= |A| - |R[A]|, \]
where $ |R[A]| $ denotes the number of edges in $ A. $

Let $ \kplusomega $ be the class of finite graphs given by
\[ \kplusomega:=\Big\{A : |A|<\aleph_{0}\text{ and } \delta(B)>0 \text{ for every non-empty } B\subseteq A  \Big\}. \]

For each $ n\geq 1, $ let $ \kplusn $ be the class of all members of $ \kplusomega $ satisfying the following axiom
\[ \forall x\bigg[\deg(x)\geq n+2\longrightarrow \neg\exists y_{1}\ldots y_{2n}\bigg(R(x,y_{1})\wedge\bigwedge_{1\leq i<j\leq2n} y_{i}\neq y_{j}\wedge\bigwedge_{i=2}^{2n}R(y_{i-1},y_{i})\bigg)\bigg].  \]

In other words, if $ \deg(x)\geq n + 2, $ then there is no path of length $ 2n + 1 $ starting at $ x. $

Finally, for $ n=0, $ let
 \[ \mathcal{K}_{0}:=\Big\{A\in\kplusomega : \deg(a)\leq3 \text{ for all } a\in A\Big\}. \]

As a convention, we assume that all classes introduced above contain the empty set. For each $ \alpha\in\omega+1, $ we denote by $ \kplusalphabar $ the class of all $ \mathcal{L} $-structures $ M $ whose finite substructures lie in $ \kplusalpha $, namely $ \age(M)\subseteq\kplusalpha. $

We recall some basic definitions and facts that are standard in the context of Hrushovski constructions. The reader might refer to \cite{Wagner-Relational, Baldwin&Shi-StableGen, Pourmahd-SimpleGen, Pourmahd-SmoothClasses} and \cite{Vali&Pourmahd-nonACGenerics} for further analysis.

	\begin{definition}\label{dfnCloseness} For $ A,B\in\kplusalpha,$
	\begin{itemize}
		\item[(i)] We say that $ A $ is \textit{closed} or \textit{strong} in $ B $ and in notations we write $ \closed{B}, $ if  $ A\subseteq B $ and for any $ C\subseteq B $ with $ A\subsetneq C $ we have that $ \delta(C)>\delta(A). $
		\item[(ii)]  We say that $ A $ is \textit{weakly closed} or \textit{weakly strong} in $ B $ and in notations we write $ A\leq B, $ if  $ A\subseteq B $ and for any $ C\subseteq B $ with $ A\subseteq C $ we have that $ \delta(C)\geq\delta(A). $
		\item[(iii)] For $ M\in\kplusalphabar $ and $ A\subseteq_{\omega} M $ we say that $ A $ is closed in $ M, $ denoted by $ A\leq^{*}M, $ if for any $ B\subseteq_{\omega} M $ with $ A\subseteq B $ we have that $ A\leq^{*}B. $ We may define $ A\leq M $ in a similar way.
		\item[(iv)] If $ A,B\subseteq D\in\kplusalpha, $ then the \textit{relative predimension} of $ B $ over $ A $ in $ D $ is defined as $ \delta(B/A):=\delta(AB)-\delta(A). $
		
	\end{itemize}
\end{definition}

The following lemma shows that the class $ \langle\kplusalpha,\leq^{*}\rangle $ possesses a natural graph theoretic interpretation.
\begin{lemma}\label{lmaGeneralProperties}
	Suppose that $ A $ is a finite $ \mathcal{L} $-structure.
		\begin{itemize}
			\item[(i)] $ A\in\kplusomega\Leftrightarrow $ the number of edges in $ A $ is strictly less than the number of vertices of $ A\Leftrightarrow A $ is an acyclic graph.
			\item[(ii)] If $ A\in\kplusalpha $ has $ k $ many connected components, then $ \delta(A)=k. $
			\item[(iii)] $ A\leq^{*}B\in\kplusalpha $ if and only if $ B\backslash A $ is not connected to $ A. $
		\end{itemize}
\end{lemma}
\begin{proof}
(i) If $ \bar{c}\subseteq A $ is a cycle, then $ \delta(\bar{c})=0 $ which implies that $ A\not\in\kplusomega. $ (ii) This is immediately followed by the fact that in a finite tree, the number of vertices equals the number of edges plus one. (iii) If $ b\in B\backslash A $ is connected to $ A, $ then $ \delta(b/A)\leq 0 $ which contradicts $ A\leq^{*}B. $
\end{proof}
\begin{definition}\label{dfnMinimal} For $ A,B\in\kplusalpha, $
\begin{itemize}
\item[(i)] $ (A,B) $ is called a \textit{minimal pair}, denoted by $ \minpair{B}, $ if  $ A\subseteq B $ and $ A $ is closed in any proper substructure of $ B $ containing $ A $ but is not closed in $ B. $ If $ \delta(B/A)=0, $ we call $ (A,B) $ a $ 0 $-minimal pair.
\item[(ii)] $ B $ is called an \textit{intrinsic extension} of $ A, $ denoted by $ \intrext{B}, $ if $ A\subseteq B $ and no proper substructure of $ B $ that contains $ A $ is closed in $ B. $ If $ \delta(B/A)=0, $ we call $ B $ a $ 0 $-intrinsic extension of $ A $.
\item[(iii)] For $ N\subseteq M\in\kplusalphabar $ the \textit{closure} of $ N $ in $ M $ is defined as the following
\[ \cl^{*}_{M}(N)=\bigcup\Big\{E\subsetfinite M\Big|\intrext{E} \text{ for some } A\subsetfinite N\Big\}. \]
\end{itemize}
The corresponding notations for weak closedness (\Cref{dfnCloseness}) are denoted respectively by $ A\not\leq_{\min}B, A\leq_{\idot}B $ and $ \cl_{M}(N). $
\end{definition}
Based on the predimension, we can also define the \textit{dimension} function as the following.
\begin{definition}\label{dfnDimension}
Suppose that $ A, B\subseteq_{\omega}M\in\kplusalphabar. $ 
\begin{itemize}
\item[(i)] The \textit{dimension} of $ A $ in $ M $ is defined as 
\[ \dMe_{M}(A):=\inf\{\delta(C)|A\subseteq C\subseteq_{\omega} M\}. \]
\item[(ii)] The \textit{relative} dimension of $ B $ over $ A $ (with respect to $ M $) is defined as
\[ \dMe_{M}(B/A):=\dMe_{M}(AB)-\dMe_{M}(A). \]
\item[(iii)] If $ N $ is an arbitrary substructure of $ M, $ then the relative dimension of $ A $ over $ N $ is defined as the following
\[ \dMe_{M}(A/N):=\inf\{\dMe_{M}(A/B)|B\subseteq_{\omega} N\}. \]
\end{itemize}
\end{definition}
\begin{lemma}\label{lmaMinimalPairs}
	Suppose that $ A,B\in\kplusalpha. $
	\begin{itemize}
	\item[(i)] If $ \minpair[A]{B}, $ then $ \delta(B/A)\leq 0 $ and for any $ C $ with $ A\subsetneq C\subsetneq B $ we have that $ \delta(C/A)> 0. $
	\item[(ii)] If $ \intrext[A]{B}, $ then $ B $ is the union of a chain of minimal pairs as the following
		\[ A=\minpair[B_{0}]{B_{1}}\minpair[]{}\cdots\minpair[]{B_{n}}=B. \]
	\item[(iii)] If $ A, B_{1}, B_{2}\subseteq M $ with $ \intrext[A]{B_{1}} $ and $ \intrext[A]{B_{2}}, $ then we have that $ \intrext[A]{B_{1}B_{2}} $.
	\item[(iv)] If $ \intrext[A]{B}, $ then for any $ C $ with $ A\subseteq C\subseteq B $ we have that $ \delta(B/C)\leq 0. $
	\end{itemize}
	Similar facts hold for $ \leq $ and its corresponding notions, when we replace $ \text{(i)}, \text{(iv)} $ by $ \text{(i)}', \text{(iv)}' $ below:
	\begin{itemize}
		\item[$ \text{(i)}' $] If $ A\not\leq_{\min}B, $ then $ \delta(B/A)< 0 $ and for any $ C $ with $ A\subsetneq C\subsetneq B $ we have that $ \delta(C/A)\geq 0. $
		\item[$ \text{(iv)}' $] If $  A\leq_{\idot}B, $ then for any $ C $ with $ A\subseteq C\subseteq B $ we have that $ \delta(B/C)< 0. $
		\end{itemize}
\end{lemma}
\begin{proof}
(i),(iv),$ \text{(i)}' $ and $ \text{(iv)}' $ are immediate from the definitions. (ii) Since $ A $ is not closed in $ B, $ there exists a minimal substructure of $ B $ containing $ A $ in which $ A $ is not closed; call it $ B_{1} $ and continue this process until covering $ B $ entirely. (iii) If it is not the case, then there is some $ C\subsetneq B_{1}B_{2} $ containing $ A $ that is closed in $ B_{1}B_{2}. $ Then, we have that $ C\cap B_{i}\leq^{*}B_{i}, $ for $ i=1,2. $ This leads to a contradiction.
\end{proof}
\begin{lemma}\label{lmaMinPair}
	Suppose that $ \minpair{B}\subsetfinite M\in\kplusalphabar, $ then $ B\backslash A $ is a singleton. Moreover, we have the following.
		\begin{itemize}
			\item[(i)] If $ \delta(B/A)=0, $ then $ B $ consists of a single element connected to $ A $ with only one edge.
			\item[(ii)] If $ \delta(B/A)<0, $ then $ B $ is a singleton with at least two relations to $ A $ and the number of distinct copies of $ B $ over $ A $ in $ M $ is 1.
		\end{itemize}
\end{lemma}
\begin{proof}
	If $ B\backslash A $ has more than one element, then by the definition of a minimal pair, for each $ b\in B\backslash A $ we have that $ A\leq^{*}Ab. $ Hence, by part (iii) of \Cref{lmaGeneralProperties}, there is no relation between $ b $ and $ A. $ Consequently, there does not exist any relation between $ B\backslash A $ and $ A, $ which implies $ A\leq^{*}B. $ This contradicts the fact that $ \minpair[A]{B}. $ Furthermore, if $ \delta(B/A)<0, $ then the fact that $ M $ can not contain a cycle implies that $ M $ has only one copy of $ B $ over $ A. $ 
\end{proof}
\begin{corollary}\label{crlWeakClFinite}
For every $ A\subsetfinite M\in\kplusalphabar, $ we have that $ \cl_{M}(A) $ is a finite structure.
\end{corollary}
\begin{proof}
Use the weak version of part (iii) in \Cref{dfnMinimal}, parts (ii), (iii) of \Cref{lmaMinimalPairs} and part (ii) of \Cref{lmaMinPair}. 
\end{proof}
\begin{lemma}\label{lmaGraphsClosurePath}
	For every $ \bar{a},b\in M\in\kplusalphabar, $ $ b\in\cl^{*}_{M}(\bar{a}) $ if and only if there is a path from $ b $ to $ \bar{a}. $
\end{lemma}
\begin{proof}
	By \Cref{lmaMinimalPairs}, any intrinsic extension is built by a finite tower of minimal pairs. Hence, by \Cref{lmaMinPair}, it is obvious that any element in $ \cl^{*}_{M}(\bar{a}) $ is connected to $ \bar{a}. $ Moreover, a path is a finite chain of zero minimal extensions, hence any element that is connected to $ \bar{a} $ lies in its closure.
\end{proof}
The next corollary follows easily from \Cref{lmaGraphsClosurePath}
\begin{corollary}\label{crlGraphsClosurePathNew}
If $ M\in\kplusalphabar, $ and $ \bar{b}, A\subsetfinite M $ with $ A\leq^{*}M $ and $ A\cap\bar{b}=\varnothing, $ then we have that	 $ A\cap\cl^{*}_{M}(\bar{b})=\varnothing. $
\end{corollary}

\begin{lemma}\label{lmaGraphsClosureUniquePath}
	Suppose that $ \bar{a},b\in M\in\kplusalphabar $ with $ \bar{a}\leq M. $ Then $ b\in\cl^{*}_{M}(\bar{a}) $ if and only if there is a unique path $ \bar{p} $ from $ b $ to $ \bar{a} $ with the property that $ \card[\bar{p}\cap\bar{a}]=1. $
\end{lemma}
\begin{proof}
	Suppose that there exist two such paths, say $ \bar{p} $ and $ \bar{p}'. $ By an  induction on the length of $ \bar{p}\cup\bar{p}', $ one can show that there is a subset $ B\subseteq\bar{p}\cup\bar{p}' $ with $ \delta(B/\bar{a})<0. $ This contradicts the fact that $ \bar{a}\leq M. $
\end{proof}
In the following corollary, we collect some easy characterizations of the introduced concepts above.
\begin{corollary}\label{crlEasyCharac}
Let $ A,B\in\kplusalpha. $
\begin{itemize}
\item[(i)] $ \intrext{B} $ if and only if $ A\not\leq^{*} B $ and every connected component of $ B\backslash A $ is attached to $ A. $
\item[(ii)] $ (A,B) $ is a weakly minimal pair if and only if $ B\backslash A=\{b_{1},\ldots, b_{m}\} $ and there are $ a,a' $ in different connected components of $ A $ such that $ ab_{1}\cdots b_{m}a' $ is a path in $ B. $ In other words, if $ B\backslash A $ contains at least one path connecting two different connected components of $ A. $
\item[(iii)] $ A\leq B $ if and only if every path in $ B $ between elements of $ A $ is contained in $ A. $
\item[(iv)] $ \intrextw{B} $ if and only if there are paths $ \bar{p}_{1},\ldots,\bar{p}_{k} $ such that $ B=A\bar{p}_{1}\cdots\bar{p}_{k} $ and:
\begin{itemize}
\item[(1)] $ \bar{p}_{1} $ connects different connected components of $ A_{1}:=A. $
\item[(2)] For $ i=2,\ldots, k, $ the path $ \bar{p}_{i} $ connects different connected components of $ A_{i}:=A_{i-1}\bar{p}_{i-1}. $
\end{itemize}
\end{itemize}
\end{corollary}
\begin{definition}\label{dfnFreeJoin}
Suppose that $ M_{0}, M_{1}, M_{2}\in\kplusalphabar $ with $ M_{1}\cap M_{2} = M_{0}. $ The structure $ M $ is called the \textit{free join} or \textit{free amalgam of} $ M_{1} $ \textit{and} $ M_{2} $ \textit{over} $ M_{0} $, denoted by $ \freeJoin{M_{1}}{M_{0}}{M_{2}}, $ if the universe of $ M $ is $ M_{1}\cup M_{2} $ and the following holds 
\[ R^{M} = R^{M_{1}}\cup R^{M_{2}}. \]
\end{definition}
It is easy to see that the class $ \langle\kplusalpha,\leq^{*}\rangle $ has the \textit{full amalgamation property}, i.e. for every $ A,B,C\in\kplusalpha $ with $ A=B\cap C $ and $ A\leq^{*}B, $ the free join of $ B $ and $ C $ over $ A $ is in $ \kplusalpha $ and we have that $ C\leq^{*}\freeJoin{B}{A}{C}. $  Hence, for the class $ \kplusalpha $ there exists a unique countable generic structure $ \mathfrak{M}_{\alpha}. $ In fact, the following  properties characterize $ \mathfrak{M}_{\alpha} $ among all countable structures.
\begin{itemize}
\item[(i)] For every $ A\in\kplusalpha,$ there is a $ \leq^{*} $-closed embedding of $ A $ into $ \mathfrak{M}_{\alpha}. $ (\textit{Universality})
\item[(ii)] For every $ A\leq^{*}B\in\kplusalpha $ with $ A\leq^{*}\mathfrak{M}_{\alpha}, $ there is a $ \leq^{*} $-closed embedding of $ B $ over $ A $ in $ \mathfrak{M}_{\alpha}. $ (\textit{Ultra-homogeneity})
\item[(iii)]  $ \mathfrak{M}_{\alpha} $ is the union of a chain of finite structures $ \{A_{i}:i\in\omega\}, $ where for each $ i\in\omega $ we have that $ A_{i}\in\kplusalpha $ and $ A_{i}\leq^{*} A_{i+1}. $ (\textit{Finite closure property}) 
\end{itemize}
\begin{lemma}\label{lmaGenericityFOExpressible}
For each $ m\in\omega, $ there is a formula $ \gamma_{m}^{*}(\bar{x}) $ with $ \card[\bar{x}]=m $ such that for every $ M\in\kplusalphabar $ and $ \bar{a}\in M $ we have the following
\begin{flalign*}
M\models \gamma_{m}^{*}(\bar{a}) \Leftrightarrow \bar{a}\leq^{*}M.
\end{flalign*} 
\end{lemma}
\begin{proof}
By \Cref{lmaMinPair}, every minimal pair over $ \bar{a} $ consists of a single point with at least one relation to $ \bar{a}. $ Hence, being closed in $ M $ is equivalent to non-existence of such a point. Namely, $ \gamma_{m}^{*}(\bar{x}) $ is the following formula
\[ \forall y\Big(\bigwedge_{i=1}^{m}y\neq x_{i}\rightarrow\bigwedge_{i=1}^{m}\neg R(x_{i},y)\Big). \]
\end{proof}
\begin{definition}\label{dfnTUniv}
Let $ \univ_{\alpha} $ be the collection of the sentences asserting that the relation $ R $ defines an acyclic graph together with the axioms of the class $ \kplusalpha $ and the following set of sentences that ensure universality
\[ \Big\{\exists\bar{x}(\diag_{A}(\bar{x})\wedge \gamma_{m}^{*}(\bar{x}))\Big\}_{A\in\kplusalpha}. \]
\end{definition}

\begin{lemma}\label{lmaGenericSatTUniv}
We have the following.
\begin{itemize}
\item[(i)] Every model of $ \univ_{\alpha} $ is ultra-homogeneous. 
\item[(ii)] $ \mathfrak{M}_{\alpha}\models \univ_{\alpha}. $
\end{itemize}
 
\end{lemma}
\begin{proof}
(i) Suppose that $ M\models\univ_{\alpha}, A\leq^{*}M $ and $ A\leq^{*}B\in\kplusalpha. $ It is enough to find a copy of $ B $ in $ M $ that is disconnected from $ A. $ Let $ m=\card[A] $ and $ B' $ be the structure that is obtained from $ m+1 $ copies of $ B $ being mutually freely amalgamated over the empty set. Using the axioms of $ \univ_{\alpha}, $ there is a closed embedding of $ B' $ into $ M. $ Hence, there is at least one copy of $ B $ in $ B' $ that is disconnected from $ A. $ Part (ii) follows from universality of $ \mathfrak{M}_{\alpha}. $
\end{proof}

\section{\bf Closure Formulas}\label{secCLF}

A key step in our approach is to introduce the notion of a \textit{closure formula} and to show that in $ \theory(\mathfrak{M}_{\alpha}), $ for each $ \alpha\in\omega+1, $ all formulas are equivalent to closure formulas. From now on, we denote $ \theory(\mathfrak{M}_{\alpha}) $ by $ T_{\alpha}. $
\begin{definition}\label{dfnPClosureformulas}
The set of closure formulas $ \clforms_{\alpha} $ is the least class of $\mathcal{L}$-formulas that is defined inductively as follows.
\begin{itemize}
\item[(i)] For each $ A\in\kplusalpha $ we let $\diag_{A}(\bar{x})\in\clforms_{\alpha}. $ 
\item[(ii)] If $\varphi_{AB}(\bar{x}\bar{y})\in\clforms_{\alpha}, $ and $ \intrext[A]{B}\in\kplusalpha, $ then the formula $ \psi_{A}(\bar{x}) $ that is in the form of  $\exists\bar{y}\Big[\diag_{(A,B)}(\bar{x},\bar{y})\wedge\varphi_{AB}(\bar{x}\bar{y})\Big] $ is in $ \clforms_{\alpha}. $
\item[(iii)] If $\varphi_{AB}(\bar{x}\bar{y})\in\clforms_{\alpha}, $ and $ \intrext[A]{B}\in\kplusalpha, $ then the formula $ \psi_{A}(\bar{x}) $ that is in the form of  $\forall\bar{y}\Big[\diag_{(A,B)}(\bar{x},\bar{y})\rightarrow\varphi_{AB}(\bar{x}\bar{y})\Big] $ is in $ \clforms_{\alpha}. $
\item[(iv)] If $\varphi_{A}, \psi_{B}\in\clforms_{\alpha},$ their Boolean combinations are also in $ \clforms_{\alpha}. $ 
\end{itemize}
\end{definition}
In short, $ \clforms_{\alpha} $ consists of those formulas whose quantifiers are relativized to closures. For more detailed information on closure formulas, we refer the reader to \cite{Vali&Pourmahd-nonACGenerics}.
\begin{definition}\label{dfnClosureType}
For a tuple $ \bar{a}\in M\in\kplusalphabar, $ the \textit{closure type} of $ \bar{a} $ in $ M $ is denoted by $ \cltype^{M}(\bar{a}) $ and is defined as the following
\[ \cltype^{M}(\bar{a}):=\Big\{\varphi(\bar{x})\in\clforms_{\alpha}\Big| M\models\varphi(\bar{a})\Big\}. \]
\end{definition}
\begin{lemma}\label{lmaClImpliesClType}
	Suppose that $ M,M'\in\kplusalphabar $ with $  \bar{a}\in M, \bar{a}'\in M' $ and $ \cl^{*}_{M}(\bar{a})\underset{\hspace*{2.5pt}\bar{a}\mapsto\bar{a}'}{\cong}\cl^{*}_{M'}(\bar{a}'). $ Then we have that
	\[ \cltype^{M}(\bar{a})=\cltype^{M'}(\bar{a}'). \]
\end{lemma}
\begin{proof}
	By induction on the complexity of formulas in $ \cltype^{M}(\bar{a}). $
\end{proof}
\begin{lemma}\label{lmaGraphsClosureFreejoin}
	Suppose that $ \bar{a},\bar{b}\in M\in\kplusalphabar $ with $ \bar{b}\cap\cl^{*}_{M}(\bar{a})=\varnothing. $ Then
	\begin{itemize}
		\item[(i)] $ \cl^{*}_{M}(\bar{a}\bar{b})=\freeJoin{\cl^{*}_{M}(\bar{a})}{\varnothing}{\cl^{*}_{M}(\bar{b})}. $
		\item[(ii)] $ \cltype^{M}(\bar{a}\bar{b}) $ is determined by $ \cltype^{M}(\bar{a}), \cltype^{M}(\bar{b}) $ and the fact that ``there is no path between $ \bar{b} $ and $ \bar{a} $''.
	\end{itemize}
	
\end{lemma}
\begin{proof}
	(i) First note that $ \cl^{*}_{M}(\bar{a})\cap\cl^{*}_{M}(\bar{b})=\varnothing. $ Since otherwise, by \Cref{lmaGraphsClosurePath}, there exists a path from $ \bar{a} $ to $ \bar{b} $ contradicting the fact that $ \bar{b}\cap\cl^{*}_{M}(\bar{a})=\varnothing. $ On the other hand, by \Cref{lmaMinPair}, any minimal pair over $ \freeJoin{\cl^{*}_{M}(\bar{a})}{\varnothing}{\cl^{*}_{M}(\bar{b})} $ is a singleton that is connected to $ \cl^{*}_{M}(\bar{a}) $ or $ \cl^{*}_{M}(\bar{b}) $ or to the both. But this contradicts the fact that both $ \cl^{*}_{M}(\bar{a}) $ and $ \cl^{*}_{M}(\bar{b}) $ are closed in $ M. $ Hence $ \freeJoin{\cl^{*}_{M}(\bar{a})}{\varnothing}{\cl^{*}_{M}(\bar{b})}\leq^{*}M $ which completes the proof. 
	
	(ii) This is proved by induction on the complexity of closure formulas. The cases of quantifier free formulas and Boolean combinations are easy to verify and the case of universal formulas follows from the Boolean and the existential cases.
	
	For the existential quantifier, suppose that there are tuples $ \bar{a},\bar{b},\bar{a}' $ and $ \bar{b}' $ with $ \cltype^{M}(\bar{a})=\cltype^{M}(\bar{a}'), \cltype^{M}(\bar{b})=\cltype^{M}(\bar{b}') $ and with no path connecting $ \bar{a} $ to $ \bar{b} $ or $ \bar{a}' $ to $ \bar{b}'. $ Moreover, by part (i), we have that $ \cl^{*}_{M}(\bar{a}\bar{b})=\freeJoin{\cl^{*}_{M}(\bar{a})}{\varnothing}{\cl^{*}_{M}(\bar{b})} $ and $ \cl^{*}_{M}(\bar{a}'\bar{b}')=\freeJoin{\cl^{*}_{M}(\bar{a}')}{\varnothing}{\cl^{*}_{M}(\bar{b}')}. $ If there is $ \bar{c}\in\cl^{*}_{M}(\bar{a}\bar{b}) $ satisfying $ \varphi(\bar{a}\bar{b},\bar{c}) $ (of complexity at most $ m $), then $ \bar{c} $ can be partitioned into tuples $ \bar{c}_{1} $ and $ \bar{c}_{2} $ in such a way that $ \cl^{*}_{M}(\bar{a}\bar{c}_{1}\bar{b}\bar{c}_{2})=\freeJoin{\cl^{*}_{M}(\bar{a}\bar{c}_{1})}{\varnothing}{\cl^{*}_{M}(\bar{b}\bar{c}_{2})}. $ Working in an $ \aleph_{0} $-saturated elementary extension $ N $ of $ M, $ since $ \bar{a}' $ and $ \bar{b}' $ respectively have the same closure types as $ \bar{a} $ and $ \bar{b}, $ one can find tuples $ \bar{d}_{1} $ and $ \bar{d}_{2} $ with the same properties as $ \bar{c}_{1} $ and $ \bar{c}_{2} $ and in such a way that $ \cltype^{N}(\bar{a}\bar{c}_{1})=\cltype^{N}(\bar{a}'\bar{d}_{1}) $ and $ \cltype^{N}(\bar{b}\bar{c}_{2})=\cltype^{N}(\bar{b}'\bar{d}_{2}). $ Hence, by applying induction hypothesis inside $ N, $ we have that $ N\models\varphi(\bar{a}'\bar{b}',\bar{d}). $ Therefore, $ M\models\exists\bar{y}\varphi(\bar{a}'\bar{b}',\bar{y}). $
\end{proof}

\begin{lemma}\label{lmaDifferenceinClosure} 
Suppose that $ A, B\subseteq M\in\kplusalphabar $ and $ b\in M. $ If $ b\in\cl^{*}_{M}(AB)\backslash\cl^{*}_{M}(B), $ then $ b\in\cl^{*}_{M}(A). $
\end{lemma}
\begin{proof}
Follows immediately from \Cref{lmaGraphsClosurePath}.
\end{proof}
\begin{lemma}\label{lmaFiniteStructure}
For every closure formula $ \varphi_{A}(\bar{x}) $ that is consistent with $ T_{\alpha}, $ there exists a finite structure $ B_{\varphi} $ with $ \intrext[A]{B_{\varphi}} $ such that for every $ M \models T_{\alpha}, $ with $ A\cong\bar{a}\subseteq M, $ if there is a closed embedding of $ B_{\varphi} $ into $ M $ over $ \bar{a}, $ we have that $ M\models\varphi_{A}(\bar{a}). $
\end{lemma}
\begin{proof}
Since $ \varphi_{A}(\bar{x}) $ is consistent with $ T_{\alpha}, $ there exists $ \bar{a}_{0}\subsetfinite \mathfrak{M}_{\alpha} $ with $ \bar{a}_{0}\cong A $ and $ \mathfrak{M}_{\alpha}\models\varphi(\bar{a}_{0}). $ Let $ B_{\varphi}=\cl^{*}_{\mathfrak{M}_{\alpha}}(\bar{a}_{0}). $ By an induction on the complexity of closure formulas, one can show that $ B_{\varphi}\models\varphi(\bar{a}_{0}). $ On the other hand, for any $ \bar{a}\subsetfinite M\models T_{\alpha} $ with $ \bar{a}\cong A, $ and a closed embedding of $ B_{\varphi} $ into $ M, $ we have that $ \cl^{*}_{M}(\bar{a})\cong_{A}B_{\varphi}. $ Hence, by \Cref{lmaClImpliesClType}, $ M\models\varphi_{A}(\bar{a}). $
\end{proof}

\begin{theorem}\label{theoremQE}
$ T_{\alpha} $ admits quantifier elimination down to closure formulas. More precisely, in $ T_{\alpha} $ every formula is equivalent to a closure formula.
\end{theorem}
\begin{proof}
Suppose that $ M $ is an $ \aleph_{0} $-saturated model of $ T_{\alpha}. $ We show that the following set defines a back-and-forth system inside $ M; $ this leads to the desired quantifier elimination.
\begin{align*}
  \mathcal{I}:= \Biggl\{(\bar{a},\bar{a}')\;\Bigg|\; \bar{a},\bar{a}'\in M,\card[\bar{a}]=\card[\bar{a}']\neq 0, \cltype^{M}(\bar{a})=\cltype^{M}(\bar{a}')\Biggr\}. \\
\end{align*}

Note that, based on \Cref{dfnPClosureformulas,dfnClosureType}, equivalence of closure types implies the equivalence of quantifier free types.
	
Now, suppose that $ \cltype^{M}(\bar{a})=\cltype^{M}(\bar{a}'), $ and $ b\in M. $ According to \Cref{crlWeakClFinite}, the weak closures of $ \bar{a} $ and $ \bar{a}' $ are finite. Hence, the fact that the closure types of $ \bar{a} $ and $ \bar{a}' $ are identical implies that $ \cl_{M}(\bar{a})\cong_{\bar{a}\mapsto\bar{a}'}\cl_{M}(\bar{a}'). $ This fact also implies that $ \cltype^{M}(\cl_{M}(\bar{a}))=\cltype^{M}(\cl_{M}(\bar{a}')). $ Therefore, we can assume that $ \bar{a} $ and $ \bar{a}' $ are weakly closed in $ M. $
	
If $ b\in\cl^{*}_{M}(\bar{a}), $ finding a suitable $ b' $ over $ \bar{a}' $ is guaranteed by $ \aleph_{0} $-saturation of $ M. $ 
	
Hence, we suppose that $ b\not\in\cl^{*}_{M}(\bar{a}). $ Using $ \aleph_{0} $-saturation of $ M $ and part (ii) of \Cref{lmaGraphsClosureFreejoin}, in order to find an element $ b' $ over $ \bar{a}' $ with $ \cltype^{M}(\bar{a}'b')=\cltype^{M}(\bar{a}b), $ we only need to find an element $ b' $ that realizes a given $ \varphi(y)\in\cltype^{M}(b) $ without having a path connecting it to $ \bar{a}'. $  

By \Cref{lmaFiniteStructure}, there exists a finite $ B_{\varphi} $ with $ \intrext[b]{B_{\varphi}} $ such that every closed embedding of $ B_{\varphi} $ into $ M $ guarantees that $ M\models\exists y\varphi(y). $ By \Cref{lmaGenericSatTUniv}, we have that $ M $ is ultra-homogeneous. Therefore there exists such a closed embedding $ f:B_{\varphi}\longrightarrow M. $ Moreover, by \Cref{crlGraphsClosurePathNew}, we have that $ fB_{\varphi} $ does not intersect $ \cl^{*}_{M}(\bar{a}') $ which means that there is no path between $ fb $ and $ \bar{a}'. $ This completes the proof.
\end{proof}
\begin{corollary}\label{crlQE}
In a sufficiently saturated model of $ T_{\alpha}, $ for every small set $ A $ and tuples $ \bar{a} $ and $ \bar{a}' $ we have that 
$ \bar{a}\equiv_{A}\bar{a}' $ if and only if
\[ \cl^{*}(\bar{a}A)\underset{\bar{a}\mapsto\bar{a}'}{\cong_{A}}\cl^{*}(\bar{a}'A). \]
\end{corollary}
\section{\bf Superstability}\label{secForking}

In this section, for each $ \alpha\in\omega+1 $ we show that $ \theory(\mathfrak{M}_{\alpha}) $ is strictly superstable. Furthermore, we show that $ \theory(\mathfrak{M}_{\omega}) $ is not 1-based of U-rank $ \omega, $ while $ \theory(\mathfrak{M}_{n}) $ is trivial of U-rank $ n+1 $ for $ n\in\omega. $

We fix a monster model for $ T_{\alpha} $ and denote it by $ \mathbb{M}_{\alpha}. $  Finite tuples in $ \mathbb{M}_{\alpha} $ are shown by $ \bar{a},\bar{b},\ldots, $ and $ A,B,\ldots $ are small subsets of $ \mathbb{M}_{\alpha}. $ For better readability, we drop the subscript $ \mathbb{M}_{\alpha} $ from all notations, hence, for example we use $ \dMe(A) $ and $ \cl^{*}(A) $ instead of $ \dMe_{\mathbb{M}_{\alpha}}(A) $ and $ \cl^{*}_{\mathbb{M}_{\alpha}}(A). $ 

Superstability for these structures, can be obtained directly by proving a uniqueness property for d-independence (\Cref{dfnDIndependence}). But some more general results on \textit{ultraflat} graphs imply superstability very easily. We recall some definitions and facts that can be found in \cite{Podewski&Ziegler-StableGraphs}, \cite{Herre-SuperstableGraphs} and \cite{Ivanov-SuperflatGraphs}.

\begin{definition}\label{dfnComplGraphDivided}
For $ r,m\geq 2, $ by $ \mathcal{C}^{r}_{m} $ we denote the class of all graphs that are obtained from the complete graph $ K_{m} $ by dividing each edge with at most $ r $ new vertices.
\end{definition}
\begin{definition}\label{dfnUltraflat}
A graph $ G $ is called \textit{ultraflat} if there exists some $ m $ such that for every $ r, $ the graph $ G $ does not contain a subgraph that is isomorphic to a member of $ \mathcal{C}^{r}_{m}. $ Here, notice the difference between the notion of a subgraph and that of an \textit{induced} subgraph.
\end{definition}
\begin{fact}(Theorem 1 of \cite{Herre-SuperstableGraphs})\label{factSuperStableGraphs}
If a graph $ G $ is ultraflat, then it is superstable.
\end{fact}

\begin{notation}
For $ a,b\in\mathbb{M}_{\alpha} $ let $ \dist(a,b) $ denote the length of the minimal path that connects $ a $ to $ b $ (If $ a $ and $ b $ are in separate connected components, set $ \dist(a,b)=\infty $). Recall that the $ r $-neighbourhood of a vertex $ a $ in a graph $ M $ is the set of all vertices $ b $ with $ \dist(a,b)\leq r, $ we denote this set by $ \Nr^{r}[a]. $
\end{notation}
 We recall the following definition from \cite{Goode_SomeTrivial}. 
 \begin{definition}\label{dfnTriviality}
  A stable theory $ T $ is called \textit{trivial}, if in any model $ M\models T, $ any set $ A\subseteq M, $ and any elements $ a,b,c\in M $ that are pairwise independent (in the sense of non-forking) over $ A, $ we have that $ \forkindep{a}{Ac}{b}. $
  \end{definition}
 \begin{theorem}\label{theoremSuperStability}
 For $ \alpha\in\omega+1, $ the theory $ T_{\alpha} $ is strictly superstable and trivial.
 \end{theorem}
 \begin{proof}
Note that any forest is ultraflat because it does not contain any subgraph isomorphic to a member of $ \mathcal{C}_{3}^{r} $ for any $ r\in\omega. $ Hence, by \Cref{factSuperStableGraphs}, $ T_{\alpha} $ is superstable.

Recall that a theory is small if it has at most countably many types over the empty set. Now, for each $ I\subseteq\omega, $ let $ p_{I}(x) $ be the type that is defined as the following.
 \begin{itemize}
\item[(1)] If $ I=\varnothing, $ then $ ``\deg(x)=0"\in p_{I}(x). $ Otherwise, 
\item[(2)] if $ 0\in I, $ then $ ``\deg(x)=2"\in p_{I}(x) , $ otherwise $ ``\deg(x)=1"\in p_{I}(x). $
\item[(3)] If $ i+1\in I, $ then the following formula is in $ p_{I}(x) $
 \[ \forall y[\dist(x,y)=i+1\longrightarrow \deg(y)=3], \]
 otherwise, $ p_{I}(x) $ contains the following formula
 \[ \forall y[\dist(x,y)=i+1\longrightarrow \deg(y)=2]. \]
 \end{itemize}
 Literally, $ p_{I}(x) $ asserts the existence of a tree $ \mathbb{T}_{I}, $ with an infinite height that is rooted at $ x $ such that if $ i\in I, $ then all the elements at the $ i $th level of $ \mathbb{T}_{I} $ have degree $ 3, $ otherwise they are of degree $ 2. $ It is obvious that if $ I\neq J, $ then $ p_{I}(x)\neq p_{J}(x) $ and $ \mathbb{T}_{I}\not\cong\mathbb{T}_{J}. $ Note that, by \Cref{lmaGenericSatTUniv}, we have that $ T_{\alpha}\supseteq\univ_{\alpha}, $ hence, it is easy to see that $ p_{I}(x) $ is consistent with $ T_{\alpha}. $ Therefore, the theory $ T_{\alpha} $ is not small, hence not $ \omega $-stable. Triviality is a consequence of Theorem 1.4 of \cite{Ivanov-SuperflatGraphs} and the fact that \textit{monadic} stability is equivalent to tree decomposability (\cite{Baldwin&Shelah-SecondOrderQuantifiers}).
 \end{proof}
 
 We recall the following definitions and facts from \cite{Ivanov-SuperflatGraphs}.
 \begin{definition}\label{dfnIvanovComponent}
 Suppose that $ G $ is a graph and $ A $ is a subset of $ G. $
 \begin{itemize}
 \item[(i)] We say that two elements $ b_{1}, b_{2}\in G $ are  \textit{connected over} A if they are connected by a path disjoint from $ A. $
 \item[(ii)] If $ a\in G, $ the \textit{component} of $ a $ over $ A $ in $ G, $ denoted by $ C_{G}(a/A), $ is the set of all $ b\in G\backslash A $ connected with $ a $ over $ A. $
 \end{itemize}
 \end{definition}
 \begin{remark}\label{rmkComponentIsClosure}
 Note that if $ a\in M\in\kplusalphabar,A\subseteq M $ and $ a\not\in\cl^{*}_{M}(A), $ then by \Cref{lmaGraphsClosurePath} we have that $ C(a/A)=\cl^{*}_{M}(a). $
 \end{remark}
 \begin{fact}(Lemma 2.1 of \cite{Ivanov-SuperflatGraphs})\label{factIvanovForking}
Suppose that $ A=\acl(A)\subseteq B. $ Then $ \type(\bar{a}/B) $ does not fork over $ A $ if and only if for every $ a\in\bar{a}, $ we have that $ C(a/A)\cap B=\varnothing. $
 \end{fact}

Using \Cref{lmaGraphsClosureUniquePath} and the fact above, the following lemma, provides a more concrete description of forking inside $ \mathbb{M}_{\alpha}. $
 \begin{lemma}\label{lmaForking}
 Suppose that $ a\in\mathbb{M}_{\alpha}, $ and $ A\leq B\leq\mathbb{M}_{\alpha}. $ Then, $ \type(a/B) $ forks over $ A $ if and only if either of the following cases occur:
 \begin{itemize}
 \item[(1)] $ a\in\cl^{*}(B)\backslash\cl^{*}(A). $
 \item[(2)] $ a\in(\cl^{*}(A)\backslash\acl(A))\cap B. $
 \item[(3)] $ a\in\cl^{*}(A)\backslash(B\acl(A)) $ and if $ \bar{p} $ is the unique path that connects $ a $ to $ A $ with $ \card[\bar{p}\cap A]=1, $ then $ \bar{p}\cap (B\backslash A)\neq\varnothing $ and $ \bar{p}\cap(B\backslash A) $ is not algebraic over $ A. $ (Note that, by \Cref{lmaGraphsClosureUniquePath}, there exists such unique $ \bar{p}. $)
 \end{itemize}
 \end{lemma}
 
   \begin{definition}\label{dfnDIndependence}
   Suppose that $ \bar{b} $ is a finite tuple and $ A $ and $ C $ are subsets of $ \mathbb{M}_{\alpha} $.
   \begin{itemize}
   \item[(i)]  We say that $ \bar{b} $ is \textit{d-independent} of $ A $ over $ C $ and, in notations, we write $ \indep[\dMe]{\bar{b}}{C}{A}, $ if the following hold
   \begin{itemize}
   \item[$ (\text{a}) $]  $ \dMe(\bar{b}/C)=\dMe(\bar{b}/AC) $
   \item[$ (\text{b}) $] $ \cl(\bar{b}C)\cap\cl(AC)=\cl(C). $
   \end{itemize}
   		
   \item[(ii)] We say that $ B $ is d-independent of $ A $ over $ C, $ in notations $ \indep[\dMe]{B}{C}{A,} $ if every finite subset of $ B $ is d-independent of $ A $ over $ C. $ 
   \end{itemize}
   \end{definition}
     
    The following fact is well known in the literature, one can refer to \cite{Baldwin&Shi-StableGen} and \cite{Wagner-Relational} for more details.
     
   \begin{fact}\label{factDIndEQUIVFreeJoin}
   Suppose that $ A $ and $ B $ are weakly closed in $ M\in\kplusalphabar $ with $ A\cap B=C. $ The following are equivalent
   \begin{itemize}
   \item[(i)] $ \indep[\dMe]{A}{C}{B}. $ 
   \item[(ii)] $ \cl(AB)=\freeJoin{A}{C}{B}. $ More precisely, $ AB\leq M $ and $ A $ and $ B $ are in free amalgam in $ M $ over $ C. $
   \end{itemize}
   \end{fact}

  Using \Cref{factIvanovForking} and \Cref{factDIndEQUIVFreeJoin} one can see that in $ T_{\alpha}, $ non-forking coincides with the notion of d-independence over the algebraically closed sets. Hence, the following theorem is established.
  
  \begin{theorem}\label{theoremForkEqualDIndep}
  Suppose that $ \bar{b} $ is a finite tuple and $ A $ and $ C $ are subsets of $ \mathbb{M}_{\alpha} $ with $ \acl(A)=A. $ Then $ \indep[\dMe]{\bar{b}}{A}{C} $ if and only if $ \forkindep{\bar{b}}{A}{C}. $
  \end{theorem} 
 \begin{notation}\label{notInfiniteBranchTree}
 Let $ \mathfrak{p}(x) $ be the type expressing that for each $ m\geq1 $ and every $ y, $ if $ \dist(x,y)\leq m, $ then the degree of $ y $ is infinite.
 \end{notation}
 
 \begin{remark}\label{rmkInfiniteBranchingTree}
It is easy to see that if an element $ a $ in a structure $ M\in\kplusomegabar $ realizes $ \mathfrak{p}(x), $ then $ \cl_{M}(a) $ is an infinite-branching tree of infinite height which we denote by $ \mathbb{T}_{\infty}. $ Furthermore, if $ M\models T_{\omega}, $ then, by \Cref{theoremQE}, every two elements of $ \mathbb{T}_{\infty} $ have the same type in $ M, $ hence, any (finite) path in $ \mathbb{T}_{\infty} $ is algebraically closed.
 \end{remark}
 
 \begin{theorem}\label{theoremOmegaRank}
 $ T_{\omega} $ is of U-rank $ \omega. $ Moreover, it is not 1-based.
 \end{theorem}
 \begin{proof} 
We first show that for each element $ a, $ the U-rank of $ \type(a) $ is at most $ \omega. $ To this end, we show that for every weakly closed $ B $ with $ \type(a/B) $ forking over the empty set, $ \type(a/B) $ has a finite U-rank. Based on \Cref{lmaForking}, the element $ a $ must be in $ \cl^{*}(B). $ By \Cref{lmaGraphsClosureUniquePath}, there is a unique path $ \bar{p}, $ that connects $ a $ to $ B $ with $ \card[\bar{p}\cap B]=1. $ By induction on $ m=\card[\bar{p}] $ and using the Lascar inequality, we can see that the U-rank of $ a $ over $ B $ is less than or equal to $ m. $
 
Moreover, the type $ \mathfrak{p}(x), $ introduced above, is finitely satisfiable in $ \mathfrak{M}_{\omega}. $ Hence, there is a realization $ a\in\mathbb{M}_{\omega} $ for  $ \mathfrak{p}(\bar{x}). $ For each $ n\geq 1$ one can find an element $ b_{n}\in\cl^{*}(a)=\mathbb{T}_{\infty} $ that is connected to $ a $ by a path of length $ n, $ say $ b_{n}b_{n-1}\ldots b_{1}a. $ By Case 1 in \Cref{lmaForking}, the type $ \type(a/b_{n}) $ forks over the empty set. Moreover, by \Cref{rmkInfiniteBranchingTree}, for each $ 1< i\leq n, $ we have that $ a,b_{i-1}\not\in\acl(b_{i}\cdots b_{n}). $ Hence, using Case 3 in \Cref{lmaForking}, the following
  \[ \mathfrak{p}(x)=\type(a/\varnothing)\subset\type(a/b_{n})\subset\type(a/b_{n-1}b_{n})\subset\cdots\subset\type(a/b_{1}\cdots b_{n}) \]
 is a forking chain of length $ n. $ This proves that $ T_{\omega} $ has U-rank $ \omega. $
 
Now, suppose that $ b\in\mathbb{T}_{\infty} $ with $ \mathbb{M}_{\omega}\models R(a,b). $ Since $ \mathbb{T}_{\infty}\leq^{*}\mathbb{M}_{\omega}, $ we have that $ \acl(b)=b. $ For a subset $ E $ containing $ b $ with $ \forkindep{a'}{b}{E}, $ by \Cref{factIvanovForking}, we have that $ \cl^{*}(aE)=\freeJoin{\cl^{*}(ab)}{b}{\cl^{*}(E)}. $ Hence, by part (ii) of  \Cref{lmaGraphsClosureFreejoin}, the type $ q(x)=\type(a/b) $ is stationary. Let $ \Nr^{1}[b] $ be the collection of all elements $ z\in\mathbb{T}_{\infty} $ with $ \dist(z,b)=1. $ We have that $ q(\mathbb{M}_{\omega})=\Nr^{1}[b]. $ On the other hand, since $ \mathbb{M}_{\omega} $ does not have a cycle, for every $ f\in\aut(\mathbb{M}_{\omega}) $ we have that $ f $ fixes $ \Nr^{1}[b] $ setwise if and only if it fixes $ b. $ Hence, the canonical base of $ q $ is $ \dcl^{\eq}(b). $ However, by \Cref{rmkInfiniteBranchingTree}, $ b $ is not algebraic over $ a $ and therefore, $ T_{\omega} $ is not 1-based.
 \end{proof}
  
  \begin{theorem}\label{theoremFiniteRank}
  For every $ n\in\omega, $ the theory $ T_{n} $ has U-rank $ n+1 $ and is 1-based.
  \end{theorem}
  \begin{proof}
   For $ n=0, $ note that for any $ A\subseteq M\in\overline{\mathcal{K}_{1}} $ we have that $ \cl^{*}_{M}(A)=\acl_{M}(A). $ Now working in $ \mathbb{M}_{1}, $ for an element $ a $ and sets $ A\subset B $ if $ a\not\in\cl^{*}(B), $ by \Cref{lmaForking}, we have that $ \forkindep{a}{A}{B}. $ On the other hand, if $ a\in\cl^{*}(B), $ then $ a $ is algebraic over $ B $ and any superset of $ B $ cannot yield a forking extension for $ \type(a/B). $  Hence, the U-rank is always less than or equal to one. Moreover, when $ a\in\cl^{*}(B) $ the following is an example of a forking chain of length 1
        \[ \type(a)\subset\type(a/B). \]
  
  For each $ n\geq1, $ first we construct a structure $ \langle A_{n},a,b_{1},\ldots,b_{n}\rangle. $ Let $ A_{1} $ be a structure consisting of an element $ b_{1} $ with infinitely many copies of an element $ a $ all being connected to $ b_{1} $ by an edge. Let $ \langle A_{2},a,b_{1},b_{2}\rangle $ be a structure with a new element $ b_{2} $ that is connected to $ b_{1} $ together with infinitely many copies of $ A_{1} $ over $ b_{2} $ all having the same type over $ b_{2}. $ Proceed inductively, having $ \langle A_{n-1},a,b_{1},\ldots,b_{n-1}\rangle $ for $ n\geq3, $ let $ \langle A_{n},a,b_{1},\ldots,b_{n}\rangle $ be the structure that is obtained by adding a new element $ b_{n} $ connected to $ b_{n-1} $ together with infinitely many copies of $ \langle A_{n-1},a,b_{1},\ldots,b_{n-1}\rangle $ over $ b_{n} $ all having the same type as $ \langle A_{n-1},a,b_{1},\ldots,b_{n-1}\rangle $ over $ b_{n}. $ The following diagram displays $ \langle A_{2},a,b_{1},b_{2}\rangle. $
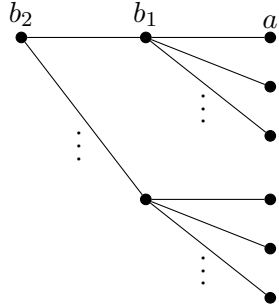
  \begin{mytikzpicture}
  	\node[label=above:$ b_{1} $](b1){};
  	\node[right=1.5cm of b1,label=above:$ a $](a){};
  	\node[below= .5cm of a](a11){};
  	\node[below= .5cm of a11](a12){};
  	
  	\node[left=1.5cm of b1,label=above:$ b_{2} $](b2){};
  
  	\node[below =2cm of b1](b21){};
  	\node[right=1.5cm of b21](a20){};
  	\node[below= .5cm of a20](a21){};
  	\node[below= .5cm of a21](a22){};
  	
  	\tikzstyle{every node} = [];
  	\node[below right=1.25cm and .5cm of b1,rotate=90](d1){$ \cdots $};
  	\node[below right=1.25cm and .5cm of b21,rotate=90](d2){$ \cdots $};
  	\node[below right=1.75cm and .5cm of b2,rotate=90](d3){$ \cdots $};
  	
  	\draw (b2) -- (b1) -- (a)
  		  (b1) -- (a11)
  		  (b1) -- (a12)
  		  (b2) -- (b21) -- (a20)
  		  (b21) -- (a21)
  		  (b21) -- (a22);
  \end{mytikzpicture}
Note that $ \langle A_{n},a,b_{1},\ldots,b_{n}\rangle\in\kplusnbar, $ hence $ \mathfrak{M}_{n} $ realizes every finite part of it. Hence, by describing the closure type of $ ab_{1},\ldots,b_{n} $ in $ A_{n} $ and using \Cref{theoremQE}, one can realize $ \langle A_{n},a,b_{1},\ldots,b_{n}\rangle $ in $ \mathbb{M}_{n} $ such that for each $ 1\leq i\leq n, $ the path $ ab_{1}\cdots b_{i-1} $ is not algebraic over $ b_{i+1}\cdots b_{n}. $ Hence, using Case 3 in \Cref{lmaForking}, the following is a forking chain of length $ n+1 $
  \[ \type(a/\varnothing)\subset\type(a/b_{n})\subset\type(a/b_{n-1}b_{n})\subset\cdots\subset\type(a/b_{1}\cdots b_{n})\subset\type(a/ab_{1}\cdots b_{n}). \]

Now, suppose that in $ \mathbb{M}_{n}, $ there is an element $ a $ and subsets $ B_{1}\subseteq\cdots\subseteq B_{n+2} $ that form a forking chain of length $ n+2 $ for the non-algebraic type $ \type(a/B_{1}). $ Using \Cref{lmaForking}, there must be a path $ \bar{p} $ that connects $ a $ to $ B_{1} $ such that for each $ 1\leq i\leq n+2 ,$ the path $ \bar{p} $ intersects $ \acl(B_{i})\backslash\acl(B_{i-1}). $ Hence, $ \bar{p}\cap\acl(B_{n+2}) $ has at least $ n+1 $ many elements. For each $ 1\leq i\leq n+2, $ let $ b_{i} $ be the element in $ \bar{p}\cap \acl(B_{i})\backslash\acl(B_{i-1}) $ with the minimum distance to $ a. $ Note that we might have $ b_{n+2}=a. $

For each $ 2\leq i\leq n+2, $ the element $ b_{i} $ is not algebraic over $ b_{i-1}. $ Therefore, there exists a distinct copy of $ \bar{p}\backslash\acl(B_{1}) $ over $ b_{1}. $ Let denote this path by $ \bar{p}'; $ also, for each $ i\geq 2, $ denote the corresponding elements by $ b'_{i}. $ Hence, there exists a path connected to $ b_{n+1} $ and extending $ b_{n}\cdots b_{1}b'_{2}\cdots b'_{n+2} $ which is of length at least $ 2n+1. $ On the other hand, since obviously $ \bar{p}\backslash\acl(B_{n+1}) $ is not algebraic over $ b_{n+1}, $ we have that $ \deg(b_{n+1})=\infty. $ This contradicts the axioms of $ \kplusn $ and proves that U-rank equals $ n+1 $ in $ T_{n}. $

1-basedness is obtained by Proposition 9 in \cite{Goode_SomeTrivial} and the fact that $ T_{n} $ is superstable, trivial and of finite U-rank.
\end{proof}

\begin{remark}
The results in this section can be obtained without using ultraflat graphs. More precisely, one can directly prove uniqueness for d-independence as well as the extension property over the algebraically closed sets. Then, using uniqueness for d-independence, superstability can be obtained by directly counting the types. \Cref{lmaForking} can also be proved by uniqueness as well.
\end{remark}

\section{\bf Pseudofiniteness and Decidability}\label{secPSF}
In this section, we prove that for each positive natural number $ n, $ the theory $ T_{n} $ is pseudofinite and decidable. Recall that a complete theory $ T $ is pseudofinite if for each $ \varphi\in T $ there exists a finite structure satisfying $ \varphi. $ This is equivalent to the fact that an ultra-product of finite structures satisfies $ T. $

The proof of \Cref{theoremPsedofinite} proceeds similar to the proof of Theorem 3.3.2 of \cite{Spencer-StrangeLogic}. Recall that a \textit{rooted tree} is a tree with a distinguished vertex $ \textbf{a} $ such that other vertices of the tree are considered to be $ \textbf{a} $'s children, grandchildren, etc. For each such tree and for a pair of positive integers $ (r,s), $ one can inductively define a function, called $ (r,s) $\textit{-value}, that provides a counting criterion determining the $ r $-neighbourhood of the root $ \textbf{a} $ by considering any degree greater than $ s $ as ``many''. The $ (r,s) $-value, is in fact a description of a finite fragment of $ \cl^{*}(\textbf{a}) $ that, to some extent, goes parallel to the way that closure formulas describe $ \cl^{*}(\textbf{a}). $

\begin{definition}\label{dfnRSValue}
Suppose that $ r,s\in\omega\backslash\{0\}. $
\begin{itemize}
\item[(i)]  The set $ \valc(r,s) $ is defined inductively as follows
\begin{itemize}
\item[(1)] $ \valc(1,s)=\{0,1,\ldots,s,\infty\}. $
\item[(2)] $ \valc(r+1,s):=\{\sigma:\valc(r,s)\longrightarrow\{0,1,\ldots,s,\infty\}\}. $ 
\end{itemize}
\item[(ii)] The $ (r,s) $-\textit{value} of a rooted tree $ \langle \textbf{T},\textbf{a}\rangle, $ denoted by $ \val_{(r,s)}\langle \textbf{T},\textbf{a}\rangle, $ is defined inductively on $ r. $ For $ r=1, $ let
\[ \val_{(1,s)}\langle \textbf{T},\textbf{a}\rangle:= 
\begin{cases}
\card[\Nr^{1}(\textbf{a})] &\quad\text{if}\quad \card[\Nr^{1}(\textbf{a})]\leq s,\\
\infty &\quad\text{otherwise.}
\end{cases}\] 
Having defined $ \val_{(r,s)}\langle \textbf{T},\textbf{a}\rangle $ for every rooted tree $\langle \textbf{T},\textbf{a}\rangle, $ we define $ \val_{(r+1,s)}\langle \textbf{T},\textbf{a}\rangle\in\valc(r+1,s) $ as follows. For every $ \sigma\in\valc(r,s), $ let  
\[ V_{\sigma}:=\{b\in\Nr^{1}(\textbf{a})|\val_{(r,s)}\langle \textbf{T},b\rangle=\sigma\}, \]  
where $ \langle \textbf{T},b\rangle $ is the subtree of $ \textbf{T} $ consisting of $ b $ as a root and all its children. Set
\[ \val_{(r+1,s)}\langle \textbf{T},\textbf{a}\rangle(\sigma):= 
\begin{cases}
\card[V_{\sigma}] &\quad\text{if}\quad \card[V_{\sigma}]\leq s,\\
\infty &\quad\text{otherwise.}
\end{cases}\] 
\end{itemize}
\end{definition} 

Now, to prove completeness for $ \univ_{n}, $ we use finite Ehrenfeucht-\fraisse  games. 
It worth noting that when a $ k $-round Ehrenfeucht-\fraisse game starts by selecting the root $ \textbf{a} $, actually all vertices of degree greater than $ k $ play the same role in the game as the vertices of degree $ k. $ This is the significance of using $ (r,k-1) $-values, for an appropriately chosen  $ r, $ to handle the situations that we encounter in such a game.

Also, recall that a \textit{Distance Ehrenfeucht-\fraisse} game is an Ehrenfeucht-\fraisse game with the additional property that the Duplicator wins only if the distance of her selected elements be the same as the distance of the Spoiler's selected elements. The $ r $-neighbourhoods of two elements $ a $ and $ b $ are called $ k $\textit{-similar} if the Duplicator wins the $ k $-round Distance Ehrenfeucht-\fraisse game that is played over their $ r $-neighbourhoods and is started by selecting $ a $ and $ b. $ We recall the following facts from \cite{Spencer-StrangeLogic}. 

\begin{fact}(Theorem 3.3.1 of \cite{Spencer-StrangeLogic})\label{factValueSimilarNeighbour}
Let $ \langle\textbf{T}_{1},\textbf{a}_{1}\rangle $ and $ \langle\textbf{T}_{2},\textbf{a}_{2}\rangle $ be two rooted trees which have the same $ (r,s-1) $-value, for some $ r $ and $ s. $ Then, $ \textbf{a}_{1} $ and $ \textbf{a}_{2} $ have $ s $-similar $ r $-neighbourhoods.
\end{fact}

\begin{fact}(Theorem 2.6.6 of \cite{Spencer-StrangeLogic})\label{factDupWinsEHR}
Suppose that $ \textbf{G}_{1} $ and $ \textbf{G}_{2} $ are two graphs and $ r=\frac{3^{k}-1}{2}. $ Moreover, suppose that
\begin{itemize}
\item[(i)] for every $ y\in \textbf{G}_{2} $ and $ x_{1},\ldots,x_{k-1}\in \textbf{G}_{1} $ there exists $ x\in \textbf{G}_{1} $ with $ x $ and $ y $ having $ k $-similar $ r $-neighbourhoods and $ \dist(x,x_{i})>2r+1 $ for $ 1\leq i\leq k-1. $ And
\item[(ii)] for every $ x\in \textbf{G}_{1} $ and $ y_{1},\ldots,y_{k-1}\in \textbf{G}_{2} $ there exists $ y\in \textbf{G}_{2} $ with $ x $ and $ y $ having $ k $-similar $ r $-neighbourhoods and $ \dist(y,y_{i})>2r+1 $ for $ 1\leq i\leq k-1. $
\end{itemize}
Then the Duplicator wins the $ k $-round Ehrenfeucht-\fraisse game played over $ \textbf{G}_{1} $ and $ \textbf{G}_{2}. $
\end{fact}

\begin{theorem}\label{theoremPsedofinite}
  For every $ 1\leq n\in\omega, $ the theory $ \univ_{n} $ is complete. Consequently, $ T_{n} $ is decidable and pseudofinite.
\end{theorem}
\begin{proof}
 Suppose that $ k $ is a positive integer and $ M_{1},M_{2}\models\univ_{n}. $ We show that the Duplicator wins the $ k $-round Ehrenfeucht-\fraisse game played over $ M_{1} $ and $ M_{2}. $ Let $ r=\frac{3^{k}-1}{2}, $ we show that the condition (i) in \Cref{factDupWinsEHR} holds for $ M_{1} $ and $ M_{2}. $ Note that $ \cl^{*}_{M_{2}}(y) $ is a possibly infinite rooted tree with root $ y. $ Using \Cref{dfnRSValue} and an induction on $ r, $ we can construct a finite rooted tree  $ \langle\textbf{T},\textbf{a}\rangle\in\kplusn $ with the same $ (r,k-1) $-value as $ \cl^{*}_{M_{2}}(y). $ By \Cref{factValueSimilarNeighbour}, we have that $ \textbf{a} $ and $ y $ have $ k $-similar $ r $-neighbourhoods.
 
 On the other hand, using \Cref{lmaGenericSatTUniv}, there is a closed embedding $f:\langle\textbf{T},\textbf{a}\rangle\longrightarrow M_{1} $ with $ x:=f(\textbf{a}) $ such that $ x $ is not connected to any of the elements $ x_{1},\ldots,x_{k-1}. $ Hence, we have that $ \dist(x,x_{i})=\infty>2r+1, $ for each $ 1\leq i\leq k-1. $ Similarly, one can show that the condition (ii) in \Cref{factDupWinsEHR} holds. Hence, by \Cref{factDupWinsEHR}, we have that $ M_{1}\equiv_{k} M_{2}. $ This shows that $ \univ_{n} $ is complete.

To verify that $ \univ_{n} $ is pseudofinite, let $ \{A_{i}\}_{i\in\omega} $ be an enumeration of the structures in $ \kplusn. $ For each $ i\in\omega, $ let $ B_{i} $ be the free amalgamation of $ A_{0},\ldots,A_{i} $ over the empty set. Now, given a non-principal ultrafilter $ \mathcal{U}, $ it can be seen that the $ \prod_{\mathcal{U}} B_{i} $ is a model of $ \univ_{n}. $
\end{proof}

\begin{remark}
Pseudofiniteness and completeness for $ \univ_{\omega} $ is a direct consequence of Theorem 3.3.2 in \cite{Spencer-StrangeLogic} that is noted by the authors in \cite{Vali&Pourmahd-PSH}.
\end{remark}
\begin{remark}\label{rmk-IkedaUltraflat}
The argument of superstability that is presented in this paper and is based on ultraflatness seems to be the standard way of analysing structures built from the graphs. This method also provides a direct proof for the superstability of the structures introduced in \cite{Ikeda-AbInitioSuperstable}.
\end{remark}

\section*{\bf Acknowledgement}
We would like to thank J. Baldwin, C. Laskowski and D. Macpherson for the helpful discussions we had during our stay at the Intitute Henri Poincar\'e (IHP). Hereby, we also would like to thank IHP and CIMPA for supporting our participation in the trimester MOCOVA 2018 held at the IHP. Also, the authors are thankful to the anonymous referee for his useful suggestions and comments.

\bibliographystyle{plain}
\bibliography{../01_AppFiles/refArticles,../01_AppFiles/refBooks}

\end{document}